\documentclass[11pt]{article}
\usepackage[a4paper,margin=26mm,top=28mm,bottom=28mm,
  headheight=14pt,headsep=8mm]{geometry}
\usepackage[T1]{fontenc}
\usepackage{lmodern}
\usepackage{mathtools,amssymb,amsthm,microtype}
\usepackage{titlesec,titling,fancyhdr}
\usepackage[hidelinks]{hyperref}
\hypersetup{
  pdftitle={The Generalized Lax Conjecture for Strictly Hyperbolic Polynomials},
  pdfauthor={Mario Kummer and Tim Netzer}}

\pretitle{\begin{center}\LARGE\bfseries}
\posttitle{\par\end{center}\vskip 0.7em}
\predate{\begin{center}\small}
\postdate{\par\end{center}}
\titleformat{\section}{\large\bfseries}{\thesection}{0.7em}{}
\titleformat{\subsection}{\normalsize\bfseries}{\thesubsection}{0.7em}{}
\titlespacing*{\section}{0pt}{2.2ex plus 1ex minus .2ex}{1.2ex plus .2ex}
\titlespacing*{\subsection}{0pt}{1.8ex plus .8ex minus .2ex}{.8ex plus .2ex}
\newcommand{\R}{\mathbb R}
\newcommand{\C}{\mathbb C}
\newcommand{\PP}{\mathbb P}
\newcommand{\inter}{\operatorname{int}}
\newcommand{\dist}{\operatorname{dist}}
\newcommand{\pr}{\operatorname{pr}}
\newcommand{\mm}{\mathbf m}

\newcommand{\conv}{\operatorname{conv}}
\newtheorem{theorem}{Theorem}
\newtheorem{lemma}{Lemma}[section]

\title{The Generalized Lax Conjecture for Strictly Hyperbolic Polynomials}
\author{%
Mario Kummer\\
\small TU Dresden
\and
Tim Netzer\\
\small University of Innsbruck}
\date{\today}
\begin{document}
\maketitle

\begin{abstract}
We report a proof found by an artificial
intelligence system for the following statement: The hyperbolicity cone of every strictly hyperbolic
polynomial is spectrahedral. This establishes the generalized
Lax conjecture for the generic case, while the full conjecture remains open.
\end{abstract}

\begin{quote}
\small
\noindent\textbf{AI Declaration.}
The results in this paper were produced almost entirely by an artificial
intelligence system.  We posed the question and suggested
possible approaches.  We evaluated the output, clarified and simplified the
arguments, checked them for correctness, and edited the final text.  There is
no settled convention for attributing such work.  We therefore do not claim
ownership and will not submit the paper to a mathematical journal.  Comments
and corrections are welcome at \href{mailto:mario.kummer@tu-dresden.de}{mario.kummer@tu-dresden.de} and
\href{mailto:tim.netzer@uibk.ac.at}{tim.netzer@uibk.ac.at}.
\end{quote}

\section{Introduction}

\subsection{Hyperbolic polynomials and the Lax conjectures}
A real homogeneous polynomial $h$ is called
\emph{hyperbolic with respect to} a vector $e$ if $h(e)>0$ and, for
every real vector $x$, the univariate polynomial $t\mapsto h(x-te)$
has only real roots.
The polynomial is \emph{strictly hyperbolic} if these roots are distinct
whenever $x$ is not a multiple of $e$. 
The associated closed \emph{hyperbolicity cone} is
\[
 C_e(h)=\{x\in\R^{n+1}\mid
       \text{all roots of }t\mapsto h(x-te)\text{ are nonnegative}\}.
\]
Equivalently, it is the closure of the connected component of
$\{h>0\}$ containing $e$. This cone is convex,
and every interior point can serve as a hyperbolicity direction
with the same cone \cite{Garding}.

For example, $h(x)=x_0^2-x_1^2-\cdots-x_n^2$ is hyperbolic with
respect to $e=(1,0,\ldots,0)$: the roots are
$x_0\pm\sqrt{x_1^2+\cdots+x_n^2}$, and its cone is the Lorentz cone 
$\{x\mid x_0\geqslant\sqrt{x_1^2+\cdots+x_n^2}\}$.
Another fundamental example is the determinant on the space of real
symmetric or complex Hermitian matrices. The roots of $\det(Z-tI)$ are the eigenvalues of
$Z$, so its hyperbolicity cone is exactly the positive semidefinite cone.

A cone is \emph{spectrahedral} if it can be written as
\[
 C=\{x\in\R^{n+1}\mid M(x)=x_0M_0+\cdots+x_nM_n\geqslant0\},
\]
where the $M_i$ are fixed complex Hermitian  matrices and $\geqslant0$
means positive semidefinite. Such a condition is a \emph{linear matrix
inequality}, and $M$ is called a linear matrix pencil.
When $M(e)>0$, the polynomial $\det M(x)$ is hyperbolic with respect
to $e$, and its hyperbolicity cone is precisely the set defined by this inequality.
This connection matters in optimization: minimizing a linear function
over an affine slice of a hyperbolicity cone is a hyperbolic program.
A spectrahedral representation makes it a semidefinite program,
for which matrix-based methods are available.

Lax's original conjecture, posed in 1958 \cite{Lax58}, concerns polynomials in
\emph{three variables}. In coordinates with $e=(1,0,0)$ and $h(e)=1$,
it asserts that every degree-$d$ hyperbolic form has a representation
$h(x)=\det(x_0I_d+x_1A+x_2B)$ with real symmetric matrices $A,B$.
This assertion follows from \cite{HV}; the connection with the original
formulation is explained in \cite{LPR}. Alternative proofs of Lax's original conjecture were given in \cite{hanselka1,hanselka2}.
An elementary approach to Hermitian determinantal representations of
hyperbolic polynomials in three variables is given in \cite{PV}.

For hyperbolic cubics in four variables, the existence of a definite
Hermitian determinantal representation in the smooth case is proved in
\cite{BuckleyKosir}. Approximation by smooth hyperbolic
cubics extends this to all hyperbolic cubics in four variables.

In more than three variables, a parameter count shows that general
hyperbolic forms of sufficiently large degree do not admit such 
determinantal representations. A natural weaker question is whether
some power of the polynomial admits such a representation:
\[
 h(x)^r=\det M(x),\qquad M(e)>0,\qquad r\geqslant1.
\]
Here $M$ is a linear matrix pencil of size $rd$ if $h$ has degree $d$.
Taking a power just repeats the roots on each line, so $C_e(h^r)=C_e(h)$. Thus a determinantal
representation of any positive power already gives a spectrahedral
representation of the original cone. It is nevertheless a stronger
requirement than asking only for a matrix inequality defining the cone.

There are several results concerning powers. In any number of
variables, every hyperbolic quadratic form has a positive power with
a determinantal representation
\cite{NTPowers}. However, taking powers does not work in
general: hyperbolic polynomials for which no positive power admits a
definite determinantal representation were constructed in
\cite{BrandenObstructions}. The underlying matroid constructions were
extended in \cite{BVY,AminiBranden}, producing further examples with
this obstruction.

A complementary approach uses necessary sums-of-squares conditions.
Such conditions are derived from the parametrized Hermite matrix in
\cite{NPTHermite} and from products of directional
derivatives in \cite{KPV}.
These conditions are unified by the notion of \emph{sos-hyperbolicity},
which yields counterexamples in every degree
at least four and every number of variables at least four, as well as
the first cubic example, in 43 variables
\cite{SaundersonNonnegativity}.
A cubic example in six variables was subsequently constructed in
\cite{BKKSS}.
For these examples, failure of the necessary sums-of-squares condition
rules out definite determinantal representations of every power.

The \emph{generalized Lax conjecture} now asks whether every hyperbolicity
cone is spectrahedral. It asks for a
representation of the \emph{cone}, not necessarily of the given
polynomial or a power of it. The preceding counterexamples therefore
do not disprove this conjecture. Indeed, the
specialized V\'amos quartic in four variables has a spectrahedral cone,
although none of its positive powers has a definite determinantal
representation \cite{KummerVamos}.
One possible route is to find an additional hyperbolic
factor $g$ such that $gh=\det M$, $M(e)>0$, and
$C_e(h)\subseteq C_e(g)$. The last condition is needed because
$C_e(gh)=C_e(h)\cap C_e(g)$: the extra factor must not cut away any
part of the desired cone.
Without this last containment condition, the existence of such a
factor $g$ and a definite determinantal representation of $gh$ is
proved in arbitrary dimension in \cite{KummerBezout},
provided that $h$ is strictly hyperbolic.

The generalized Lax conjecture is known to be true for
certain families beyond the already mentioned cases. Spectrahedrality
was proved for all
elementary symmetric polynomials, using the matrix--tree theorem
\cite{BrandenElementary}.
These results are closely related to \emph{derivative cones} \cite{renegarderi}:
if $h$ is hyperbolic with respect to $e$, then its directional
derivatives $\partial_e^k h$ are also hyperbolic for
$0\leqslant k<\deg h$, and their cones form nested relaxations
\[
 C_e(h)\subseteq C_e(\partial_e h)\subseteq\cdots
 \subseteq C_e(\partial_e^{d-1}h),\qquad d=\deg h.
\]
Spectrahedrality of the first derivative cone of a product of linear
forms defining a polyhedral cone was proved in
\cite{SanyalDerivatives}; the result extends to all higher
derivative cones of such products as well \cite{BrandenElementary}.
For the determinant on real symmetric matrices, an explicit
spectrahedral representation of the first derivative cone is given in
\cite{SaundersonDerivative}. This was extended to every derivative
cone and, by restriction, to $C_e(\partial_e^k h)$ whenever
$h=\det M$ for a real symmetric linear pencil with $M(e)>0$
\cite{KummerSpectral}.
The generalized Lax conjecture was also proved for all hyperbolic cubics
in five variables in
\cite{NetzerCubics}. On the other hand, even if the generalized Lax conjecture was true, we could not hope for small spectrahedral representations in general \cite{expobounds}.

Let us also mention a weaker form of representing convex sets.
A \emph{spectrahedral shadow} is a projection of a spectrahedron, so it allows for auxiliary variables.
If every nonzero boundary point of a hyperbolicity cone is a smooth point of the defining polynomial, then it is a spectrahedral shadow \cite{NS}. This proof uses results on lifted representations
of convex sets and convex hulls from
\cite{HeltonNieSets,HeltonNieHulls}.
This lifted conclusion was strengthened, even under the weaker
hypothesis of Nash-smooth boundary,
to a representation by second-order cones, equivalently a lifted
linear matrix inequality with blocks of size at most two
\cite{ScheidererSOCP}. 

For broader introductions to convex algebraic geometry, spectrahedra,
and semidefinite representations, we refer to
\cite{BPTBook,NetzerPlaumannBook}.

This paper presents a proof of the generalized Lax conjecture for all strictly hyperbolic polynomials:

\begin{theorem}\label{thm:main}
The hyperbolicity cone of every strictly hyperbolic polynomial is spectrahedral.
\end{theorem}

\subsection{Overview of the proof.}
Our approach is closely related to the sums-of-squares approximations
in \cite{KPV}. There, the hyperbolicity cone is realized as a linear slice of a cone of nonnegative polynomials and
nonnegativity is then replaced by a sum-of-squares condition.
Multiplication by powers of a fixed positive polynomial gives a sequence of inner approximations of the hyperbolicity cone by spectrahedral shadows.  We use the same idea, but with a different family of
polynomials. There are two main differences. First, we work with Hermitian
biforms on the complex projective variety defined by the hyperbolic polynomial. Since the Gram matrices of Hermitian biforms are unique, the resulting inner approximations are spectrahedral rather than just spectrahedral shadows.
Furthermore, we obtain a sum-of-squares certificate of one fixed degree for every boundary point, which implies that the inner approximation becomes exact.

After the initial reductions, let $h$ be irreducible, write
$C=C_e(h)$, $X\subseteq\mathbb{P}^n$ for the complex projective zero set of $h$, $\hat{X}\subseteq\C^{n+1}$ for the affine cone over $X$, and let $K$ be
the boundary of a compact base of $C$. Note that $K$ can be identified with a subset of $X(\R)$. We construct Hermitian
biforms
\[
 F_x(z,\overline z)=\sum_{i=0}^n x_i\alpha_i(z,\overline z),
\]
where the $\alpha_i$ are real-valued and
$\sum_i z_i\alpha_i=0$ on $\hat{X}$. For $a\in K$, we require $F_a$ to
be positive on $\hat{X}\setminus\C a$. At $[a]\in X$, its local
quadratic term is positive definite on the tangent space of $X$. Thus the
nonnegativity is required on the complex zero set, not just on its
real points.

In a fixed homogeneous basis
modulo $h$, a Hermitian biform on $\hat{X}$ has a unique coefficient
matrix, by polarization. The matrix therefore depends linearly on
$x$. For real sums of squares, a Gram matrix is generally not unique:
the condition is that there exists a positive semidefinite Gram
matrix satisfying the coefficient equations. Projecting out these
additional variables gives a spectrahedral shadow. In our setting
there are no additional variables. Uniqueness here is an algebraic
fact; it does not follow from nonnegativity, nor does nonnegativity
alone imply that the coefficient matrix is positive semidefinite.

To construct the biforms, we first construct a continuous field of
linear functionals on $X$. Near the real locus we use the B\'ezoutian
construction in \cite{KummerBezout}. The curvature of the boundary
gives the required positive quadratic terms. Away from the real
locus, separation supplies functionals which annihilate $z$ and
are strictly positive on $C\setminus\{0\}$. These choices can be
glued together.

We then use Stone--Weierstrass to make the field polynomial. Uniform
approximation alone would not preserve the zeros or the annihilation
identity. We therefore retain the explicit field as a fixed part
and write the correction as $\Pi^2v$, with $v$ continuous. Here
$\Pi^2$ is a scaled projection onto the real functionals annihilating
$z$; it vanishes to fourth order at the real locus. Only $v$ is
approximated, by biforms divided by powers of
$p=\sum_i|z_i|^2$. Applying $\Pi^2$ preserves the annihilation
identity and the quadratic terms exactly. The resulting error is
small compared with the quadratic lower bound near the zeros, and
strict positivity elsewhere is preserved by compactness. Clearing
the denominators gives the desired biforms.

It remains to make their coefficient matrices positive semidefinite.
The factorization theorem in \cite{DAngelo} requires
strict positivity, so it cannot be applied directly to $F_a$.
Instead, we write $F_a=u_a^*H_au_a$ on $\hat{X}$, where the entries of
$u_a$ generate the ideal of the line $\C a$. The positive quadratic
term at $[a]$ allows us to add a positive semidefinite correction
to $H_a$, without changing $u_a^*H_au_a$ on $\hat{X}$, so that the
resulting matrix $D_a$ is positive definite for every $z\ne0$.
The theorem now gives
\[
 p^ND_a=B_a^*B_a,\qquad p^NF_a=\|B_au_a\|^2\quad\text{on }\hat{X}.
\]
The spanning assertion in the theorem gives exactly the
one-dimensional kernel at $a$. For fixed $N$, this
full-spanning factorization condition on $D_a$ is open, and it is
preserved when $N$ is increased. Since the matrices $D_a$ form a
compact family, one exponent works for all $a\in K$.

The coefficient matrices of $p^NF_x$ thus form a linear pencil
$M(x)$ which is positive semidefinite on $C$ and singular on its
boundary. The kernels at distinct points of $K$ are
distinct, which implies $M(e)>0$. A segment argument then shows
that $M(x)$ cannot be positive semidefinite outside $C$. We obtain
the whole cone, not just an inner approximation. This does not
assert finite convergence of the original conditions in \cite{KPV},
since the biforms used here are different.

\section{The proof}
\subsection{Reductions and notation}\label{sec:reductions}

Real factors inherit strict
hyperbolicity, and the cone of a product is the intersection of the
factor cones. It therefore suffices to consider $h$ irreducible over $\R$, with
$h(e)=1$, $d=\deg h>1$, and $n\geqslant2$. Put $C=C_e(h)$.
We use its convexity and hyperbolicity in every interior direction
from \cite{Garding}.

Strict hyperbolicity gives $\partial_eh(b)\ne0$ at every nonzero real
zero $b$, so these points are smooth. Also $h$ is irreducible over $\C$:
otherwise real irreducibility forces $h=cg\overline g$, making every
real zero singular. Nonzero real zeros exist by hyperbolicity.

In coordinates $e=(1,0)$, let $\lambda_1(x')<\cdots<\lambda_d(x')$
be the roots of $h(t,x')$ for $x'\ne0$, and put $\lambda_j(0)=0$.
Then
\[
 C=\{(t,x')\mid t\geqslant\lambda_d(x')\},\qquad
 \lambda_j(-x')=-\lambda_{d+1-j}(x').
\]
These formulas show that $C$ is pointed: membership in $C\cap(-C)$
would give $\lambda_d(x')\leqslant t\leqslant\lambda_1(x')$ when
$x'\ne0$, and gives $t=0$ when $x'=0$.
We may thus choose coordinates
$e=(1,0)$ with compact base $B=C\cap\{x_0=1\}$.
Set $K=\partial B$ relative to this hyperplane; $B=\conv K$ by taking
lines through its interior points. Write
\[
 \hat{X}=\{z\in\C^{n+1}\mid h(z)=0\},\qquad X=\PP(\hat{X})\subseteq\PP^n,
 \qquad p(z)=\sum_{i=0}^n|z_i|^2.
\]
We identify $K$ with its image in $X(\R)$, always representing $a\in K$
by the vector with $a_0=1$. Dots denote bilinear pairings, without
conjugation; $*$ denotes conjugate transpose. We identify real vectors
and covectors using the Euclidean inner product. All projective
distances are measured in the ambient projective space using a fixed
smooth Riemannian metric.

\subsection{The geometry of the zero set}

On each real three-space $e\in W$, the results from 
\cite{HV,PV} give a pencil $T$ with
\[
 h|_W=\kappa\det T,\qquad T(e)>0,\qquad
 C\cap W=\{x\in W\mid T(x)\geqslant0\},
\]
where $\kappa>0$ because $h(e)=1$.

We first exclude lines in the zero set through boundary points,
which will give strictness in the supporting inequalities.

\begin{lemma}\label{lem:noline}
No real projective line contained in $X$ meets $K$.
\end{lemma}
\begin{proof}
The real-zero polynomial $f(x')=h(1,x')$ has compact rigidly convex
set $B$, and every point of $K=\partial B$ is smooth.
By \cite[Remark~3.2]{NS}, $f$ cannot vanish identically on an affine
line through a point of $K$. This shows the claim.
\end{proof}

The next lemma identifies the real normal functionals and shows that
each can vanish on at most one boundary ray.

\begin{lemma}\label{lem:normals}
For $[b]\in X(\R)$ put
\[
 \nu_b=\frac{\nabla h(b)}{\partial_eh(b)},\qquad
 \eta(b)=\partial_eh(b)\nabla h(b).
\]
Then $\nu_b$ is nonnegative on $C$. For $a\in C\setminus\{0\}$,
\[
 \nu_b\cdot a=0
 \quad\Longleftrightarrow\quad a\in\partial C\ \hbox{and}\ [a]=[b].
\]
In particular $\eta(b)\cdot a>0$ for $a\in K$ with $[a]\ne[b]$.
\end{lemma}
\begin{proof}
Nonnegativity follows from \cite[Lemma~2.6]{KPV} and continuity
to $\partial C$. Since $\nu_b\cdot e=1$, the nonzero covector $\nu_b$
is strictly positive on $\inter C$.

If $\nu_b\cdot a=0$, then $a\in\partial C$. Suppose $[a]\ne[b]$.
Choose a three-space containing $e,a,b$ and a pencil as above. Jacobi's formula and
$\partial_eh(b)\ne0$ imply that $T(b)$ has rank $d-1$. Its adjugate
therefore has the form $\rho vv^*$, with $\rho\ne0$ and
$0\ne v\in\ker T(b)$. Again by Jacobi's formula,
\[
 0=\partial_a h(b)=\kappa\rho\,v^*T(a)v.
\]
Since $T(a)\geqslant0$ this
gives $T(a)v=0$.
The pencil is singular on the entire span of $a,b$, giving a line
forbidden by Lemma~\ref{lem:noline}.
Conversely, proportionality and Euler's identity give
$\nabla h(b)\cdot a=0$. Finally
$\eta(b)=(\partial_eh(b))^2\nu_b$ has the same strict inequalities.
\end{proof}

The pairing $\eta(b)\cdot a$ is the restriction to $\hat{X}(\R)$ of the
polynomial used in the interlacer description in \cite{KPV}:
\[
 \Delta_{e,a}h=(\partial_eh)(\partial_ah)-h\partial_e\partial_ah,
 \qquad \Delta_{e,a}h(b)=\eta(b)\cdot a\quad(b\in \hat{X}(\R)\setminus\{0\}).
\]
A result in \cite{KPV} characterizes the closed hyperbolicity
cone by nonnegativity of $\Delta_{e,a}h$ on the whole real space.
Thus the normal field starts from the same positivity expression;
the strict equality case above additionally uses the no-line geometry.

By \cite[Lemma~2.4(ii)]{NS}, applied to the real-zero polynomial
$h(1,\cdot)$, smoothness of $K$ and Lemma~\ref{lem:noline} give
strict quasi-concavity at every $a\in K$, that is,
\begin{equation}\label{eq:curvature}
 \partial_u^2h(a)<0
 \quad\bigl(a\in K,\ u\ne0,\ u_0=0,\ \nabla h(a)\cdot u=0\bigr).
\end{equation}
Also $\partial_eh(a)>0$ for $a\in K$,
since $a$ is the largest simple root on its line parallel to $e$
and $h(e)=1$.

At nonreal zeros, separation supplies locally continuous linear
functionals that annihilate the zero and are strictly positive on the cone.

\begin{lemma}\label{lem:separate}
For nonreal $[z]\in X$, put
$P_z=\operatorname{span}_{\R}(\Re z,\Im z)$. Then $P_z\cap C=\{0\}$.
Moreover, $P_z^\perp$ contains a covector in $\inter C^\vee$.
Such covectors can be chosen continuously locally on the nonreal
locus.
\end{lemma}
\begin{proof}
The real and imaginary parts are linearly independent. 
If $0\ne a\in P_z\cap C$, multiplying $z$ by a nonzero complex scalar
lets us write $z=b+ia$ with $a,b$ real and independent. Since we can approximate $a$ by points from the interior of $C$ and since $h$ is hyperbolic with respect to every such interior point, the polynomial $h(ta+b)$ is a limit of real-rooted polynomials. On the other hand, $h(ta+b)$ has the non-real zero $i$. Hence it is identically zero, contradicting Lemma~\ref{lem:noline}.


The orthogonal projection of $B$ onto $P_z^\perp$ is compact and convex
and does not contain zero, since $P_z\cap C=\{0\}$ and $0\notin B$.
Strict separation therefore gives
$\ell\in P_z^\perp\cap\inter C^\vee$.
The planes $P_w$ and their orthogonal projections vary continuously
near a nonreal point. Projecting this fixed $\ell$ onto $P_w^\perp$
therefore gives the desired local choice.
\end{proof}

\subsection{A polynomial field of supporting covectors}

A $(k,k)$-biform is a polynomial in two groups of variables, homogeneous
of degree $k$ in each group. We evaluate the second group at
$\overline z$. Such a polynomial scales by $|c|^{2k}$ under $z\mapsto cz$.
Its sign and zero set are therefore well-defined on points of projective space, and
division by $p^k$ makes it an honest function on projective space.

The construction in \cite{KummerBezout} gives homogeneous polynomials $v_1,\ldots,v_N$, whose residue classes span the $k$th graded part of the homogeneous coordinate ring $\R[X]=\R[x_0,\ldots,x_n]/(h)$, for some $k\in\mathbb{N}$, and symmetric matrices $A_1,\ldots,A_n$ of size $N$ such that 
\begin{equation}\label{eq:kernel}
    (x_0I_N+x_1A_1+\cdots+x_nA_n)\cdot v=0\mod (h),
\end{equation}
where $v$ has entries $v_1,\ldots,v_N$. In particular, the vector $v$ does not vanish at any (complex) point of $X$. 
For independent sets of variables $z,w$ define
\begin{align*}
 \mathcal B(z,w)&=\partial_eh(z)\partial_eh(w)\cdot\left(v(z)^tI_Nv(w),v(z)^tA_1v(w),\ldots,v(z)^tA_nv(w)\right).
\end{align*}
Also put
\[
 \rho(z)=\sum_{i=1}^N v_i(z)v_i(\bar{z}).
\]

The next lemma shows that this construction gives a polynomial field that
extends the real normals.




\begin{lemma}\label{lem:seed}
The polynomial covector $\mathcal B$ has real coefficients, is symmetric in $z,w$, and
has bidegree $(r,r)$ with $r=k+d-1$. For every $(z,w)\in \hat{X}\times \hat{X}$, one has
\[
 \mathcal B(z,w)\cdot z=\mathcal B(z,w)\cdot w=0.
\]
For real $b\in \hat{X}(\R)$,
\[
 \mathcal B(b,b)=\rho(b)\eta(b).
\]
\end{lemma}
\begin{proof}
The statements on symmetry and degree are obvious. The next statement follows from (\ref{eq:kernel}). Finally, differentiating (\ref{eq:kernel}) shows that $B(b,b)$ is parallel to $\nabla  h(b)$. We further have
    \begin{equation*}
        \langle \mathcal{B}(b,b),e\rangle=\rho(b)\partial_e h(b)^2\text{ and }\langle\nabla h(b),e\rangle=\partial_e h(b).
    \end{equation*}
    This proves that $\mathcal{B}(b,b)=\rho(b)\partial_e h(b)\nabla  h(b)$.
\end{proof}

Define the real projective covector field on $X$:
\[
 \beta_0([z])=p(z)^{-r}\mathcal B(z,\overline z).
\]
Symmetry and real coefficients make its values real. For every
$[z]\in{X}$, one has
\[
 \beta_0([z])\cdot z=0,\qquad
 \beta_0([z])\cdot\overline z=0.
\]
On the real locus it is a positive multiple of $\nu_b$.

The next lemma shows that the curvature of the boundary makes the vanishing of the normal
pairing at each boundary point quadratic and positive definite.

\begin{lemma}\label{lem:contact}
For $a\in K$, the function $\beta_0([z])\cdot a$ on $X$ has a positive definite
Hermitian quadratic zero at $[a]$.
\end{lemma}
\begin{proof}
We work in the affine charts $z_0=w_0=1$. Near the smooth real point
$[a]$, the holomorphic implicit function theorem gives a local
parametrization $\phi$ of $X$ with $\phi(0)=a$ and
$\phi(\overline\xi)=\overline{\phi(\xi)}$: solve for one coordinate
and use the remaining coordinates, shifted by their values at $a$.
Use this parametrization on both copies of $X$ and put
\[
 F(\xi,\zeta)=a\cdot\mathcal B(\phi(\xi),\phi(\zeta)).
\]
Lemma~\ref{lem:seed} gives $F(\xi,0)=F(0,\zeta)=0$.
Thus every term of its power series contains coordinates from both
groups, so it has no linear part and its quadratic part is $\xi^TE\zeta$.
Symmetry and compatibility with conjugation make $E$ real symmetric;
on $\zeta=\overline\xi$, this quadratic part is $\zeta^*E\zeta$.

For any $0\ne s\in\R^{n-1}$, set $b(t)=\phi(ts)$ and
$u=b'(0)\ne0$. Then $u_0=0$ and $\nabla h(a)\cdot u=0$.
Taylor's formula about $b=b(t)$ gives
\[
 0=h(a)=h(b)+\nabla h(b)\cdot(a-b)
 +\tfrac12\partial_{a-b}^2h(b)+O(\|a-b\|^3).
\]
Since $h(b)=0$, Euler's identity gives
$\nabla h(b)\cdot(a-b)=\partial_a h(b)$.
Using $b(t)=a+tu+O(t^2)$, we obtain
\[
 \partial_a h(b(t))=-\tfrac12\partial_u^2h(a)t^2+O(t^3).
\]
By Lemma~\ref{lem:seed},
\[
 F(ts,ts)=\rho(b(t))\partial_eh(b(t))\partial_a h(b(t)).
\]
Comparing coefficients of $t^2$ therefore gives
\[
 s^TEs=-\tfrac12\rho(a)\partial_eh(a)\partial_u^2h(a)>0,
\]
because $\rho(a)>0$, $\partial_eh(a)>0$, and \eqref{eq:curvature}
applies to $u$. Thus $E$ is positive definite on real vectors, real symmetry gives
 positive definiteness on
complex vectors. Finally,
\[
 \beta_0([\phi(\xi)])\cdot a
 =p(\phi(\xi))^{-r}F(\xi,\overline\xi)
 =p(a)^{-r}\xi^*E\xi+O(\|\xi\|^3).
\]
Since $p(a)^{-r}>0$, this proves the assertion.
\end{proof}

Compactness turns this local quadratic positivity into a uniform
lower bound near the real locus.

\begin{lemma}\label{lem:near}
There are a neighborhood $U$ of $X(\R)$ in $X$ and $c>0$ such that
\[
 \beta_0(z)\cdot a\geqslant c\dist(z,[a])^2
 \qquad(a\in K,\ [z]\in U).
\]
\end{lemma}
\begin{proof}
The diagonal $\{(a,[a])\mid a\in K\}$ is compact. In finitely many
local charts, the positive quadratic forms of Lemma~\ref{lem:contact}
have a uniform positive lower eigenvalue bound, while their Taylor
remainders have uniform cubic bounds. Local coordinates and ambient
distance are uniformly comparable in smaller charts. Thus the stated
estimate holds whenever $\dist(z,[a])$ is sufficiently small, with
constants independent of $a$.

For real $[b]$ outside a fixed small neighborhood of $[a]$,
Lemma~\ref{lem:normals} gives $\beta_0(b)\cdot a>0$.
These pairs form a compact set, so their values have a positive
minimum. By continuity, a smaller neighborhood $U$ of the real locus
preserves half this minimum for the corresponding pairs with $[z]\in U$.
Since projective distances are bounded, this positive constant also
bounds a positive multiple of $\dist(z,[a])^2$. Combining the two
estimates proves the assertion.
\end{proof}

We now glue the field near the real locus to the separating
functionals at nonreal points to obtain a continuous field on all of $X$.

\begin{lemma}\label{lem:continuous}
There is a continuous real covector field $\beta$ on $X$ such that
$\beta([z])\cdot z=0$ for every $[z]\in X$. It equals $\beta_0$ on
a neighborhood of $X(\R)$ and belongs to $\inter C^\vee$ at every
nonreal point.
\end{lemma}
\begin{proof}
Let $U$ be the open neighborhood from Lemma~\ref{lem:near}.
At every nonreal point of $U$, that lemma makes $\beta_0$ strictly
positive on $K$, hence on $B$. Compactness of $B$ therefore gives
$\beta_0\in\inter C^\vee$ there. If $U=X$, take $\beta=\beta_0$.
Otherwise choose an open neighborhood $W$ of $X(\R)$ with
$\overline W\subseteq U$.

By Lemma~\ref{lem:separate}, the compact set $X\setminus U$ has a
finite cover by open patches $V_1,\ldots,V_m$ contained in
$X\setminus\overline W$, carrying continuous real covector fields
$\ell_j([z])\in P_z^\perp\cap\inter C^\vee$.
These fields are obtained by projecting fixed covectors $\ell_j^0$,
so $\|\ell_j([z])\|\leqslant\|\ell_j^0\|$ on $V_j$.

For the finite open cover $\{U,V_1,\ldots,V_m\}$, 
let $d_0,d_1,\ldots,d_m$ be the distances to the respective complements of $U,V_1,\ldots,V_m$ and define $$\psi_j(x)=\frac{d_j(x)}{\sum_{j=0}^md_j(x)}.$$
The $\psi_j$ are continuous functions on $X$ that
sum to one, and each is zero outside its assigned open set.
Define
\[
 \beta=\psi_0\beta_0+\sum_{j=1}^m\psi_j\ell_j,
\]
extending each $\psi_j\ell_j$ by zero outside $V_j$.
The bounds on $\ell_j$ ensure that these extensions are continuous,
so $\beta$ is a continuous real covector field.
Every contributing field annihilates $z$, hence $\beta([z])\cdot z=0$.
On $W$ all $\psi_j$ with $j\geqslant1$ vanish, so $\psi_0=1$ and
$\beta=\beta_0$. At a nonreal point, every field with nonzero weight
belongs to the convex set $\inter C^\vee$, and therefore so does
$\beta$.
\end{proof}

The following projection will preserve the identity $\beta([z])\cdot z=0$  on $X$ when approximating the continuous field $\beta$ by polynomials without changing the quadratic terms at real points.

\begin{lemma}\label{lem:projection}
For $p=\sum_i|z_i|^2$ as above we define
\[
 J=\frac{z\overline z^T-\overline z z^T}{2ip},\qquad
 \delta=-\tfrac12\operatorname{tr}(J^2),\quad\text{and }\quad \Pi=\delta I+J^2.
\]
At nonreal points, one has $\delta>0$ and
\[
 \Pi=\delta\pr_{P_z^\perp},\qquad
 \Pi^2=\delta\Pi=\delta^2\pr_{P_z^\perp}.
\]
The matrix $Q=p^2\Pi$ has real-valued
$(2,2)$-biform entries. Moreover, uniformly near $X(\R)$,
\[
 \delta=O(\dist(z,X(\R))^2),\qquad
 \Pi^2=O(\dist(z,X(\R))^4).
\]
\end{lemma}
\begin{proof}
The projection identity follows from the rank-two case of the
Pl\"ucker-to-projection formula in \cite[Corollary~2.5]{DFRS},
by taking the complementary projection and multiplying by $\delta$.
We give the elementary calculation here.
Write $z=x+iy$ and put $A=pJ=yx^T-xy^T$.
At a nonreal point, $x,y$ are independent. On their span,
the square of the skew-symmetric matrix $A$ is $-\Delta I$, where
$\Delta=\|x\|^2\|y\|^2-(x\cdot y)^2$; on its orthogonal complement
it is zero. This follows, for example, by writing
$x=se_1$, $y=te_1+ve_2$ in an orthonormal basis: the nonzero block is
$\left(\begin{smallmatrix}0&-sv\\sv&0\end{smallmatrix}\right)$.
Thus $\delta=\Delta/p^2$ and $\Pi=\delta\pr_{P_z^\perp}$.
Squaring this projection identity gives the stated formula for $\Pi^2$.
At real points, $x,y$ are dependent and $J=\delta=\Pi=0$.
Moreover,
\[
 A=\frac{z\overline z^T-\overline z z^T}{2i},\qquad
 Q=A^2-\tfrac12\operatorname{tr}(A^2)I.
\]
The entries of $A$ are real-valued $(1,1)$-biforms, so this formula
shows directly that those of $Q$ are real-valued $(2,2)$-biforms.

The entries of the matrix $J$ are smooth real valued functions on the whole ambient projective space and they
vanish on its real locus. Boundedness of their derivatives and
comparison with a nearest point of $X(\R)$ therefore give
$J=O(\dist(z,X(\R)))$, uniformly. Both $\delta$ and $\Pi$ are
quadratic in $J$, so $\delta=O(\dist(z,X(\R))^2)$ and
$\Pi^2=O(\dist(z,X(\R))^4)$, as claimed.
\end{proof}

We next show that normalized biforms provide the uniform polynomial approximations
needed to replace the continuous field.

\begin{lemma}\label{lem:approx}
Every continuous real vector-valued function $v$ on $X$ can be
approximated uniformly by $P(z,\overline z)/p(z)^j$, where all entries
of $P$ are real-valued $(j,j)$-biforms.
\end{lemma}
\begin{proof}
The real and imaginary parts of $z_i\overline z_j/p$ generate a unital
real algebra of continuous functions on $X$. They separate projective
points: the rank-one orthogonal projection $z\overline z^T/p$ determines
the line $\C z$. Stone--Weierstrass therefore gives uniform density.
Each polynomial in these generators is a sum of biforms divided by
powers of $p$. Multiplying numerators by suitable powers of $p$ gives
a common denominator $p^j$ and bidegree $(j,j)$. Apply this to each
component, again choosing a common degree.
\end{proof}

Finally, we combine this approximation with the projection in order to get a polynomial field
with the required positivity, annihilation, and quadratic vanishing.

\begin{lemma}\label{lem:alpha}
For some $k\geqslant\max(d,2)$ there is a real-valued $(k,k)$-biform
vector $\alpha(z,\overline z)$ such that on $X$ we have
\[
 z\cdot\alpha(z,\overline z)=0,
 \qquad F_a(z,\overline z)\coloneqq a\cdot\alpha(z,\overline z)>0
 \quad([z]\ne[a],\ a\in K).
\]
For every $a\in K$, the function $F_a$ has a positive definite Hermitian quadratic
zero at $[a]$.
\end{lemma}
\begin{proof}
Off the real locus put $v=(\beta-\beta_0)/\delta^2$. Extend it by zero
near the real locus, where its numerator is identically zero.
This gives a continuous function.
Both $\beta$ and $\beta_0$ take values in $P_z^\perp$, so Lemma~\ref{lem:projection} gives
$\beta=\beta_0+\Pi^2v$.

Choose 
$P$ with entries
real-valued $(j,j)$-biforms such that
 $v'=P/p^j$ satisfies $\|v'-v\|_\infty<\epsilon$ and set
$\beta'=\beta_0+\Pi^2v'$. It still satisfies $\beta'([z])\cdot z=0$, and
\[
 \|\beta'-\beta\|
  =\|\Pi^2v'-\Pi^2v \|
 \leqslant\delta^2\|v'-v\|
 \leqslant\epsilon\delta^2,
\]
where we used $\Pi^2=\delta^2\pr_{P_z^\perp}$.
Let $M=\max_{a\in K}\|a\|$. On a neighborhood where $\beta=\beta_0$,
the preceding lemmas and $\dist(z,X(\R))\leqslant\dist(z,[a])$ imply
\[
 \beta'(z)\cdot a\geqslant
 \bigl(c-\epsilon M L\bigr)\dist(z,[a])^2
\]
for a fixed $L$: a fourth power of distance is bounded by a constant
times its square on projective space. Choose $\epsilon$ so the
coefficient is positive. On the compact complement, $\beta(z)\cdot a$
has a positive minimum over all $a\in K$; decreasing $\epsilon$ preserves
that minimum as well. On the real locus the correction is zero.
The correction vanishes to fourth order, so the quadratic term at
$[a]$ is unchanged.

Finally choose $k\geqslant\max(r,j+4,d,2)$. With $Q=p^2\Pi$, define
\[
 \alpha(z,\overline z)=p^k\beta'
 =p^{k-r}\mathcal B(z,\overline z)+p^{k-j-4}Q^2P.
\]
This is a polynomial vector of the claimed bidegree. Multiplication
by $p^k$ preserves all signs and multiplies the quadratic term at
$a$ by a positive constant.
\end{proof}

\subsection{From biforms to coefficient matrices}

Polarization lets us replace identities involving complex conjugates
by identities in two independent points of $\hat{X}$.

\begin{lemma}\label{lem:polarize}
If a polynomial $G\in\C[z,w]$ satisfies $G(z,\overline z)=0$
for every $z\in \hat{X}$, then $G=0$ on $\hat{X}\times \hat{X}$.
\end{lemma}
\begin{proof}
Choose a nonzero smooth real point $0\neq b\in \hat{X}(\R)$. Since $h$ has real
coefficients, the holomorphic implicit function theorem gives a local
parametrization $\phi$ of $\hat{X}$ with $\phi(0)=b$ and
$\phi(\overline\xi)=\overline{\phi(\xi)}$. Expand
\[
 G(\phi(\xi),\phi(\zeta))
 =\sum_{\mu,\nu}c_{\mu\nu}\xi^\mu\zeta^\nu.
\]
Setting $\zeta=\overline\xi$ gives zero by hypothesis. Applying
$\partial_\xi^\mu\partial_{\overline\xi}^\nu$ at zero gives
$\mu!\nu!c_{\mu\nu}=0$. Thus $G$ vanishes on a complex open
neighborhood of $(b,b)$ in $\hat{X}\times \hat{X}$.

Since $h$ is irreducible over $\C$, $\hat{X}\times \hat{X}$ is irreducible.
A proper algebraic subset has smaller dimension and cannot contain
this neighborhood. Hence $G=0$ on $\hat{X}\times \hat{X}$.
\end{proof}

Let $A=\R[z]/(h)$, and let $\mm_j$ be a column of real homogeneous
polynomials representing a basis of its degree-$j$ part $A_j$.
In this quotient basis, polarization ensures that each restricted
biform determines a unique coefficient matrix:

\begin{lemma}\label{lem:coefficients}
Every restricted $(j,j)$-biform has a unique expression
$\mm_j(z)^*H\mm_j(z)$ on $\hat{X}$, with $H$ a complex matrix. If the
biform is real-valued, then $H$ is Hermitian. Writing $N=\dim_{\R}A_j$,
the evaluation vectors $\mm_j(z)$, for $z\in \hat{X}$, span $\C^N$
over $\C$.
\end{lemma}
\begin{proof}
Reducing each variable group modulo $h$ gives existence of the
expression. The evaluation vectors span by Hilbert's Nullstellensatz, since $h$ is irreducible over $\C$.

If $\mm_j(z)^*H\mm_j(z)=0$, Lemma~\ref{lem:polarize} gives
$\mm_j(w)^TH\mm_j(z)=0$ for all $z,w\in \hat{X}$. Spanning in both variables
implies $H=0$, proving uniqueness. For a real-valued biform, taking
the complex conjugate replaces $H$ by $H^*$, so uniqueness gives
$H=H^*$.
\end{proof}

This uniqueness distinguishes the present construction from the real
sums-of-squares relaxation in \cite{KPV}. There one seeks a positive semidefinite Gram matrix
among generally many choices, leading to a projection of a spectrahedron.
Here a restricted biform determines its Hermitian coefficient matrix
uniquely, so linear dependence on parameters passes directly to the
matrix, without introducing extra Gram-matrix variables.

The annihilation identity for the field now gives a prescribed
evaluation vector in the kernel of each boundary matrix.

\begin{lemma}\label{lem:kernel}
Write $\alpha_i(z,\overline z)=\mm_k(z)^*A_i\mm_k(z)$ and put
$S_a=\sum_i a_iA_i$. Then $S_a$ is Hermitian and
$S_a\mm_k(a)=0$ for every $a\in K$.
\end{lemma}
\begin{proof}
The $A_i$ are Hermitian by Lemma~\ref{lem:coefficients}. The polynomial
\[
 G(z,w)=\mm_k(w)^T\Bigl(\sum_i z_iA_i\Bigr)\mm_k(z)
\]
vanishes at $w=\overline z$ by $z\cdot\alpha(z,\overline z)=0$.
By Lemma~\ref{lem:polarize}, $G$ vanishes on $\hat{X}\times \hat{X}$.
Setting $z=a$ and using the spanning assertion of
Lemma~\ref{lem:coefficients} gives $S_a\mm_k(a)=0$.
\end{proof}

\subsection{Turning positivity into sums of squares}

We factor out the prescribed zero at $a$ to eventually obtain strict positivity.

\begin{lemma}\label{lem:pointideal}
For $a\in K$ put $u_i=z_i-a_i z_0$, $1\leqslant i\leqslant n$.
There are polynomial matrices and vectors with
\[
 \mm_k=z_0^k\mm_k(a)+R_a u,\qquad h=r_a^Tu.
\]
Their coefficients depend polynomially on $a$, and their degrees in
$z$ are $k-1$ and $d-1$, respectively. With $H_a=R_a^*S_aR_a$, we have
\[
 F_a(z,\overline z)=u^*H_au\quad\hbox{on }\hat{X}.
\]
\end{lemma}
\begin{proof}
Substitute $z_i=a_i z_0+u_i$ in a homogeneous polynomial $f$ of degree
$m$. The term independent of $u$ is $z_0^mf(a)$; every other term
has a factor $u_i$. Collecting these terms gives
$f=z_0^mf(a)+\sum_i R_i u_i$, with the stated degrees and polynomial
dependence on $a$. Apply this to $\mm_k$ and to $h$, using $h(a)=0$.
In $\mm_k^*S_a\mm_k$, the constant term and both cross terms vanish
by Lemma~\ref{lem:kernel} and Hermitian symmetry. The remaining term
is $u^*R_a^*S_aR_a u$.
\end{proof}

The next matrix supplies a positive semidefinite correction that
leaves $F_a$ unchanged on $X$ and controls all remaining directions.

\begin{lemma}\label{lem:correction}
Define
\[
 U_a=p^{k-2}(\|u\|^2I-uu^*)+p^{k-d}\overline r_a r_a^T.
\]
For $z\ne0$ this matrix is positive semidefinite and
$u^*U_au=p^{k-d}|h(z)|^2$. Its kernel is zero off $\hat{X}$, is $\C u(z)$
on $\hat{X}\setminus\C a$, and for $z\in\C a\setminus\{0\}$ is the
complex tangent space to $X$ at $[a]$ in the chart $z_0=1$,
identified with a subspace of $\C^n$.
Moreover, $H_a(z)$ is positive on every nonzero vector of this kernel.
\end{lemma}
\begin{proof}
For $w\in\C^n$,
\[
 w^*U_aw=p^{k-2}\bigl(\|u\|^2\|w\|^2-|u^*w|^2\bigr)
       +p^{k-d}|r_a^Tw|^2\geqslant0
\]
by Cauchy--Schwarz. Taking $w=u$ gives the scalar identity.
If $u\ne0$, equality in the first summand holds exactly when
$w\in\C u$. Such a vector also lies in the second kernel precisely
when $r_a^Tu=h(z)=0$, unless it is zero.

When $u=0$ and $z\ne0$, $[z]=[a]$. At the representative $a$, differentiation
of $h=r_a^Tu$ gives $r_a(a)=(\partial_i h(a))_{i=1}^n$.
This vector is nonzero: otherwise Euler's identity and $a_0=1$ would
also give $\partial_0h(a)=0$, contradicting smoothness.
Thus $\ker U_a(a)=\{w\mid r_a(a)^Tw=0\}$, the claimed tangent space.

For $[z]\ne[a]$ in $X$, positivity on the kernel follows from
$u^*H_au=F_a(z,\overline z)>0$. At $a$, choose a local holomorphic
parametrization $\phi$ of $X$ in the chart $z_0=1$, with $\phi(0)=a$.
Write
\[
 u(\phi(\xi))=L\xi+O(\|\xi\|^2),
\]
where $L\colon\C^{n-1}\to\C^n$ is injective with image the tangent space.
By Lemma~\ref{lem:pointideal},
\[
 F_a(\phi(\xi),\overline{\phi(\xi)})
 =\xi^*L^*H_a(a)L\xi+O(\|\xi\|^3).
\]
Lemma~\ref{lem:alpha} gives $L^*H_a(a)L>0$, so $H_a(a)$ is positive
on every nonzero tangent vector. Finally, for $c\in\C\setminus\{0\}$,
homogeneity gives
\[
 U_a(ca)=|c|^{2k-2}U_a(a),\qquad
 H_a(ca)=|c|^{2k-2}H_a(a).
\]
This proves both the kernel description and positivity for all
nonzero representatives of $[a]$.
\end{proof}

The following compactness argument makes the correction large enough
to obtain positive definiteness throughout the family.

\begin{lemma}\label{lem:compactcorrection}
Let $H_s,U_s$ be continuous Hermitian matrix families on a compact
parameter space, with $U_s\geqslant0$. If $H_s$ is positive on every
nonzero vector of $\ker U_s$, then $H_s+\lambda U_s>0$ for one
$\lambda>0$ and all $s$.
\end{lemma}
\begin{proof}
Otherwise choose $s_j$ and unit vectors $w_j$ such that
$w_j^*(H_{s_j}+jU_{s_j})w_j\leqslant0$. Pass to a convergent subsequence.
The $H_s$ are bounded, so $w_j^*U_{s_j}w_j\to0$. At the limit,
$w^*U_sw=0$, hence $w\in\ker U_s$.
But also $w_j^*H_{s_j}w_j\leqslant0$, so $w^*H_sw\leqslant0$,
contradicting the hypothesis for the unit vector $w$.
\end{proof}

Apply this lemma with $s=(a,z)\in K\times\{p(z)=1\}$. We obtain
one $\lambda>0$ such that
\[
 D_a=H_a+\lambda U_a>0\qquad(a\in K,\ z\ne0).
\]
Here homogeneity extends positivity from the unit sphere. The entries
of $D_a$ have bidegree $(k-1,k-1)$, and the coefficients depend
continuously on $a$. The scalar identity on $X$ is still
$u^*D_au=F_a$.

We use the following factorization theorem to turn positive definite
polynomial matrices into Hermitian sums of squares after multiplication by $p^N$.

\begin{theorem}[{\cite[Theorem~1]{DAngelo}}]\label{input:stabilization}
Let $D(z,\overline z)$ be an $m\times m$ Hermitian matrix with
$(s,s)$-biform entries, positive definite for all $z\ne0$. For some
integer $N\geqslant0$ there is a polynomial matrix $B(z)$ with entries
in $\C[z_0,\ldots,z_n]$, each homogeneous of degree $s+N$, such that
\[
 p(z)^ND(z,\overline z)=B(z)^*B(z),
\]
and the rows of $B$ span all $m$-tuples of homogeneous polynomials
of degree $s+N$.
\end{theorem}

Compactness and monotonicity allow the same stabilization exponent
to be used for every boundary point.

\begin{lemma}\label{lem:uniformN}
There is one $N$ such that every $D_a$ admits a factorization
$p^ND_a=B_a^*B_a$ whose rows span all polynomial $n$-tuples
homogeneous of degree $k+N-1$.
\end{lemma}
\begin{proof}
For a fixed $N$, the full-spanning factorization condition is equivalent
to positive definiteness of the coefficient matrix of $p^ND$, indexed
by degree-$(k+N-1)$ monomials in $\C[z_0,\ldots,z_n]$ and vector
coordinates \cite[Section~II]{DAngelo}.
This coefficient matrix depends linearly on $D$, so the condition
is open in the finite-dimensional coefficient space.

These conditions increase with $N$. Indeed,
$pB^*B=\sum_i(z_iB)^*(z_iB)$, and the rows $z_iB$ span the next degree,
since every monomial of positive degree is divisible by some $z_i$.

The family $\{D_a\mid a\in K\}$ is compact because its coefficients
depend continuously on $a$. The above external input places each $D_a$ in one of the open sets above. Choose a finite
subcover and take the largest of its exponents. By monotonicity,
this exponent works for every $a\in K$.
\end{proof}

Powers of the squared norm are also used to obtain sums-of-squares
inner approximations of the cone in
\cite{KPV}.
Here the multiplier is applied to a different, Hermitian construction
after resolving its prescribed zeros. The preceding lemma gives one
exponent for the entire boundary family, not only interior points.

The spanning property of the factorization makes each stabilized
boundary matrix positive semidefinite with exactly the prescribed kernel.

\begin{lemma}\label{lem:exactkernel}
Put $q=k+N$. The coefficient matrix of $p^NF_a$ in the basis $\mm_q$
is positive semidefinite, with kernel exactly $\C\mm_q(a)$.
\end{lemma}
\begin{proof}
Let $L_a$ be the constant coefficient matrix obtained by expressing
the entries of $B_a u$ modulo $h$ in the basis $\mm_q$, so that
$B_a u=L_a\mm_q$ on $\hat{X}$. Then
\[
 p^NF_a=\|B_a u\|^2
       =\mm_q^*L_a^*L_a\mm_q
 \quad\hbox{on }\hat{X}.
\]
By Lemma~\ref{lem:coefficients}, the coefficient matrix in question
is $L_a^*L_a$, which is positive semidefinite and has kernel
$\ker L_a$.

The rows of $B_a$ span over $\C$ all polynomial $n$-tuples
homogeneous of degree $q-1$. Hence the entries of $B_a u$ span
all homogeneous degree-$q$ polynomials vanishing at $a$, by the
expansion in Lemma~\ref{lem:pointideal}.
Since $h(a)=0$, their classes modulo $h$ span exactly the kernel
of evaluation at $a$ on $A_q\otimes_{\R}\C$.
This kernel has codimension one because $z_0^q(a)=1$.
Thus the row space of $L_a$ is
\[
 \{c\mid c\mm_q(a)=0\},
\]
and consequently $\ker L_a=\C\mm_q(a)$.
\end{proof}

The role of the spanning condition here parallels the generation
condition in the notion of a nice B\'ezoutian
\cite{KummerBezout}.

\subsection{Putting together the matrix pencil}

Distinct boundary points give distinct evaluation lines, so their
associated one-dimensional kernels have trivial intersection.

\begin{lemma}\label{lem:evallines}
If $a,b\in K$ are distinct, then $\C\mm_q(a)\ne\C\mm_q(b)$.
\end{lemma}
\begin{proof}
The vectors $a,b$ are not proportional because both have first
coordinate $1$. Choose a real linear form $\ell$ with $\ell(a)=0$
and $\ell(b)\ne0$.
Since $\mm_q$ represents a basis of $A_q$, there is a unique real
coefficient row $c$ such that $\ell^q\equiv c\mm_q\pmod h$.
Since $h(a)=h(b)=0$, we have
$c\mm_q(a)=0$ and $c\mm_q(b)=\ell(b)^q\ne0$.
Hence the evaluation vectors cannot be proportional over $\C$.
\end{proof}

We assemble the boundary matrices into a single linear pencil and
show that its positive semidefinite locus is exactly $C$.

\begin{lemma}\label{lem:pencil}
Let $M_i$ be the coefficient matrix of $p^N\alpha_i$ in the basis
$\mm_q$, and put $M(x)=\sum_i x_iM_i$. Then
\[
 C=\{x\in\R^{n+1}\mid M(x)\geqslant0\},\qquad M(e)>0.
\]
\end{lemma}
\begin{proof}
Linearity and coefficient uniqueness identify $M(a)$ with the matrix
of $p^NF_a$. Lemma~\ref{lem:exactkernel} gives
\[
 M(a)\geqslant0,\qquad \ker M(a)=\C\mm_q(a)\quad(a\in K).
\]
Since $B=\conv K$ and $C$ consists of its nonnegative multiples,
$M$ is positive semidefinite on $C$.
Every nonzero point of $\partial C$ is a positive multiple of a point
of $K$, so its matrix is singular. Since $M(0)=0$, $M$ is singular
on all of $\partial C$.

Since $e$ lies in the relative interior of $B$, a chord through $e$
has distinct endpoints $a,b\in K$, with $e=(1-s)a+sb$ for some
$0<s<1$. A positive combination of positive semidefinite matrices
has kernel equal to the intersection of their kernels. Hence
\[
 \ker M(e)=\ker M(a)\cap\ker M(b)
 =\C\mm_q(a)\cap\C\mm_q(b)=\{0\},
\]
where the last equality follows from Lemma~\ref{lem:evallines}.
Thus $M(e)>0$.

If $x\notin C$ but $M(x)\geqslant0$, the segment from $e$ to $x$
first meets $\partial C$ at $y=(1-t)e+tx$ with $0<t<1$. Yet
\[
 M(y)=(1-t)M(e)+tM(x)\geqslant(1-t)M(e)>0,
\]
contradicting boundary singularity. This proves the equality.
\end{proof}

We now combine the construction with the initial reductions to prove
Theorem~\ref{thm:main}.
\begin{proof}
Under the reductions in Section~\ref{sec:reductions},
Lemma~\ref{lem:alpha} constructs the polynomial supporting field.
Lemmas~\ref{lem:coefficients}--\ref{lem:exactkernel} convert it into
 matrices with one-dimensional evaluation kernels.
Lemma~\ref{lem:pencil} assembles them into a pencil defining $C$.
\end{proof}




\end{document}